\documentclass{amsart}
\usepackage{amsmath}
\usepackage{amssymb}
\usepackage{graphicx}
\newtheorem{theorem}{Theorem}[section]
\newtheorem{lemma}[theorem]{Lemma}
\newtheorem{proposition}[theorem]{Proposition}

\theoremstyle{definition}

\newtheorem{remark}[theorem]{Remark}

\numberwithin{equation}{section}

\begin{document}

\title[The geometric mean error for self-affine measures]{Asymptotic order of the geometric mean error for some self-affine measures}
\author{Sanguo Zhu}
\address{School of Mathematics and Physics, Jiangsu University
of Technology\\ Changzhou 213001, China.}
\email{sgzhu@jsut.edu.cn}

\author{Shu Zou}

\subjclass[2000]{Primary 28A75, 28A80; Secondary 94A15}
\keywords{geometric mean error, self-affine measure, Bedford-McMullen carpets}

\begin{abstract}
Let $E$ be a Bedford-McMullen carpet associated with a set of affine mappings $\{f_{ij}\}_{(i,j)\in G}$ and let $\mu$ be the self-affine measure associated with $\{f_{ij}\}_{(i,j)\in G}$ and a probability vector $(p_{ij})_{(i,j)\in G}$. We study the asymptotics of the geometric mean error in the quantization for $\mu$. Let $s_0$ be the Hausdorff dimension for $\mu$. Assuming a separation condition for $\{f_{ij}\}_{(i,j)\in G}$, we prove that the $n$th geometric error for $\mu$ is of the same order as $n^{-1/s_0}$.
\end{abstract}

\maketitle

\section{Introduction}

In this paper, we study the asymptotics of the geometric mean error in the quantization for the self-affine measures on Bedford-McMullen carpets. The quantization problem for probability measures has a deep background in information theory and engineering technology (cf. \cite{GN:98,Za:63}). One of the main objectives of this problem is to study the errors when approximating a given probability measure with discrete probability measures that are supported on finite sets. The quantization problem with respect to the geometric mean error is a limiting case of that in $L_r$-metrics as $r$ decreases to zero; it is usually more difficult than the latter because the involved integrals are typically negative and the integrands are in logarithmic forms. Also, for the above reason, those techniques which are developed for the $L_r$-quantization are often not applicable. We refer to \cite{GL:00,GL:04} for rigorous mathematical foundations of quantization. One can see \cite{GL:01,GL:04,GL:05,GL:12,KZ:16,Kr:08,LM:02,MR:15,PK:01}) for related results.
\subsection{Quantization error and quantization coefficient}

Let $\nu$  be a Borel probability measure on $\mathbb{R}^q$.
 For every $k\in\mathbb{N}$, we write
 \[
 \mathcal{D}_k:=\{\alpha\subset\mathbb{R}^q:1\leq{\rm card}(\alpha)\leq k\}.
 \]
 The $k$th quantization error for $\nu$ of order $r\in[0,\infty)$ is defined by
\begin{eqnarray}\label{quanerror}
e_{k,r}(\nu)=\left\{\begin{array}{ll}\big(\inf\limits_{\alpha\in\mathcal{D}_k}\int d(x,\alpha)^{r}d\nu(x)\big)^{1/r}&r>0\\
\inf\limits_{\alpha\in\mathcal{D}_k}\exp\big(\int\log d(x,\alpha)d\nu(x)\big)&r=0\end{array}\right..
\end{eqnarray}

By \cite{GL:04}, we have, $e_{k,r}(\nu)\to e_{k,0}(\nu)$ as $r\to 0$, provided that $\int |x|^sd\nu(x)<\infty$ for some $s>0$. Thus, the quantization for $\nu$ of order zero can be regarded as a limiting case of that of order $r>0$. We also call $e_{k,0}(\nu)$ the $k$th \emph{geometric mean error} for $\nu$.

A set $\alpha\in\mathcal{D}_k$ is called a $k$-optimal set for $\nu$ of order $r$ if  the infimum in (\ref{quanerror}) is attained at  $\alpha$. Let $C_{k,r}(\nu)$ denote the collection of all such sets $\alpha$. By \cite[Theorem 4.12]{GL:00}, for every $r>0$, $C_{k,r}(\nu)$ is non-empty whenever the $r$th moment for $\nu$ is finite:
\[
\int |x|^rd\nu(x)<\infty.
\]

In the following, we simply write $e_k(\nu)$ for $e_{k,0}(\nu)$ and write $C_k(\nu)$
for $C_{k,0}(\nu)$. By Theorem 2.5 of \cite{GL:04}, we have $e_k(\nu)>e_{k+1}(\nu)$ and $C_k(\nu)\neq\emptyset$ if the following condition is satisfied:
\begin{equation}\label{optiexists}
\int_0^1\sup_{x\in\mathbb{R}^q}\nu(B(x,s))\frac{1}{s}ds<\infty,
\end{equation}
where $B(x,s)$ denotes the closed ball of radius $s$ which is centered at $x$. This condition is fulfilled if there exist some constants $C,t$, such that
\[
\sup_{x\in\mathbb{R}^q}\nu(B(x,\epsilon))\leq C\epsilon^t\;{\rm for\;all}\;\;\epsilon>0.
\]

The upper and lower quantization coefficients are natural characterizations for the asymptotic properties of the quantization errors. Recall that for $s>0$, the $s$-dimensional upper and lower quantization coefficient for $\nu$ of order $r$ are defined by
\begin{equation}\label{coefficient}
\overline{Q}_r^s(\nu):=\limsup_{k\to\infty}k^{\frac{1}{s}}e_{k,r}(\nu),\;\underline{Q}_r^s(\nu):=\liminf_{k\to\infty}k^{\frac{1}{s}}e_{k,r}(\nu).
\end{equation}
The upper (lower) quantization dimension $\overline{D}_r(\nu) (\underline{D}_r(\nu))$ for $\nu$ of order $r$ is exactly the critical point at which $\overline{Q}_r^s(\nu) (\underline{Q}_r^s(\nu))$ "jumps" from infinity to zero:
\[
\overline{D}_r(\nu):=\limsup_{k\to\infty}\frac{\log k}{-\log e_{k,r}(\nu)},\;\underline{D}_r(\nu):=\liminf_{k\to\infty}\frac{\log k}{-\log e_{k,r}(\nu)}.
\]
One may see \cite{GL:00,GL:04,PK:01} for more details.
In comparison with the upper and lower quantization dimension, the upper and lower quantization coefficient provide us with more accurate information for the asymptotics of the quantization error. In \cite{GL:04}, Graf and Luschgy established general results on the asymptotics of the geometric mean errors for absolutely continuous distributions and self-similar measures on $\mathbb{R}^q$.
\subsection{Bedford-McMullen carpets and self-affine measures}
Let $m,n$ be integers with $n\geq m\geq2$. Let $G$ be a subset of
\begin{equation}\label{data}
\Gamma:=\{0,1,\ldots,n-1\}\times\{0,1,\ldots,m-1\}.
\end{equation}
We assume that ${\rm card}(G)\geq 2$. We consider the following affine mappings:
\[
f_{ij}:(x,y)\mapsto\bigg(\frac{x+i}{n},\frac{y+j}{m}\bigg),\;(x,y)\in\mathbb{R}^2;\;(i,j)\in G.
\]
From  \cite{Hut:81}, there exists a unique non-empty compact set $E\subset\mathbb{R}^2$ satisfying
\[
E=\bigcup_{(i,j)\in G}f_{ij}(E).
\]
The set $E$ is referred to as the self-affine set associated with $\{f_{ij}\}_{(i,j)\in G}$. We also call $E$ a Bedford-McMullen carpet.
Given a probability vector $(p_{ij})_{(i,j)\in G}$, there exists a unique Borel probability measure $\mu$ supported on $E$ such that
\begin{equation}\label{self-affinemeasure}
\mu=\sum_{(i,j)\in G}p_{ij}\mu\circ f_{ij}^{-1}.
\end{equation}
The measure $\mu$ is called the self-affine measure associated with $\{f_{ij}\}_{(i,j)\in G}$ and $(p_{ij})_{(i,j)\in G}$. Self-affine sets and self-affine measures have attracted great attention of mathematicians (cf. \cite{Bed:84,Fal:10,J:11,King:95,Mcmullen:84,Peres:94b}) in the past decades, and related problems are often rather difficult.

Let $G_y:={\rm Proj}_yG$ and $\vartheta:=\frac{\log m}{\log n}$. We define
\begin{eqnarray}
&G_x(j):=\{i: (i,j)\in G\},\;\;q_j:=\sum_{i\in G_x(j)}p_{ij},\;j\in G_y;\nonumber\\
&s_0:=-\frac{1}{\log m}\bigg(\vartheta\sum\limits_{(i,j)\in G}p_{ij}\log p_{ij}+(1-\vartheta)\sum\limits_{j\in G_y}q_j\log q_j\bigg).\label{s0}
\end{eqnarray}
By \cite{J:11,King:95}, the Hausdorff dimension for $\mu$ is equal to $s_0$. More exactly, we have
\[
\lim_{\epsilon\to 0}\frac{\log\mu(B(x,\epsilon))}{\log\epsilon}=s_0\;\;{\rm for}\;\;\mu-a.e.\;x.
\]
This, along with \cite[Corollary 2.1]{zhu:12}, implies that $\underline{D}(\mu)=\overline{D}(\mu)=s_0$. Unfortunately, this does not provide us with accurate information for the asymptotics of the geometric mean error. In order to obtain the exact asymptotic order of the geometric mean error for $\mu$, we need to examine the finiteness and positivity of the upper and lower quantization coefficient for $\mu$ of order zero. As the main result of the present paper, we will prove

\begin{theorem}\label{mthm}
Let $m\geq n\geq 3$ be fixed integers and let $\Gamma$ be as given in (\ref{data}). Let $G$ be a subset of $\Gamma$ with ${\rm card}(G)\geq 2$. Assume that
\begin{equation}\label{ssc}
\max\{|i_1-i_2|,|j_1-j_2|\}\geq 2
 \end{equation}
 for every pair of distinct words $(i_1,j_1),(i_2,j_2)\in G$. Then for the self-affine measure $\mu$ as defined in (\ref{self-affinemeasure}), we have $0<\underline{Q}_0^{s_0}(\mu)\leq\overline{Q}_0^{s_0}(\mu)<\infty$.
\end{theorem}

\begin{remark}
For $r>0$, Kessb\"{o}hmer and Zhu \cite{KZ:16} proved that $\underline{D}_r(\mu)=\underline{D}_r(\mu)$ and determined the exact value; they also proved the finiteness and positivity of the upper and lower quantization coefficient for $\mu$ of order $r$ in some special cases and Zhu \cite{zhu:18} proved this fact in general by associating subsets of $E$ with those of the product coding set $W:=G^{\mathbb{N}}\times G_y^{\mathbb{N}}$ and considering an auxiliary measure which is supported on $W$.
\end{remark}

In order to prove Theorem \ref{mthm}, we will also embed subsets of $E$ into the product coding set $W$ and consider an auxiliary product measure $\lambda$ which is supported on $W$. As is noted in \cite{zhu:18}, the images of non-overlapping rectangles (under the above-mentioned embedment) may be overlapping. This is one of the main obstacles in the way of proving the main result.

 In \cite{zhu:18}, the author removed the possible overlappings by keeping the largest of those pairwise overlapping sets and deleting smaller ones, and then estimated the possible "loss". However, in the study of the geometric mean error, the involved integrals are usually negative and the integrands are in logarithmic forms. As a consequence, the method in \cite{zhu:18} is not applicable. Our new idea here is to replace those overlappings with some other subsets of $W$ such that the sets in the final collection are pairwise disjoint. We also need to estimate the possible loss which is caused by such replacements.

The remaining part of the paper is organized as follows. In section 2, we establish an estimate for the geometric mean error for $\mu$ and reduce the asymptotics of $(e_k(\mu))_{k=1}^\infty$ to those of a number sequence $(s_k)_{k=1}^\infty$. We will need the assumption (\ref{ssc}) so that the three-step procedure as depicted in \cite{KZ:15b} can be applied. In section 3, we consider the coding space $W$ and determine the asymptotic order for two related number sequences $(d_k)_{k=1}^\infty$ and $(t_k)_{k=1}^\infty$ in terms of $s_0$. In section 4, we associate subsets of $E$ with those of $W$ and remove the possible overlappings by using the above-mentioned replacements. This enables us to establish a relationship between $(s_k)_{k=1}^\infty$ and $(t_k)_{k=1}^\infty$ and complete the proof for Theorem \ref{mthm}.
\section{Preliminaries}
We denote by $|A|$ the diameter of a set $A\subset\mathbb{R}^2$ and $A^\circ$ its interior in $\mathbb{R}^2$; let $(A)_\delta$ denote the closed $\delta$-neighborhood of $A$ for $\delta>0$. For $x\in\mathbb{R}$, let $[x]$ denote the largest integer not exceeding $x$. We will use the following notation in the remaining part of the paper (cf. \cite{GL:04}):
\begin{equation}\label{logen0}
\hat{e}_k(\nu):=\log e_k(\nu)=\inf_{\alpha\in\mathcal{D}_k}\int\log d(x,\alpha)d\nu(x),\;k\geq 1.
\end{equation}

Let $k_0:=\vartheta^{-1}$. For every $k<k_0$, we define $\Omega_k:=G_y^k$; for $k\geq k_0$ we define
\[
\ell(k):=[k\vartheta],\;\Omega_k:=G^{\ell(k)}\times G_y^{k-\ell(k)};\;|\sigma|:=k\;{\rm for}\;\sigma\in\Omega_k;\;\;\Omega^*:=\bigcup_{k=1}^\infty\Omega_k.
\]

For $\omega=((i_1,j_1),\ldots,(i_h,j_h))\in G^h$ and $(i,j)\in G$, we write
\[
\omega\ast (i,j):=((i_1,j_1),\ldots,(i_h,j_h),(i,j))\in G^{h+1}.
\]
In the same manner, we define $\rho\ast j$ for $\rho\in G_y^*$ and $j\in G_y$.

Let $\sigma=\big((i_1,j_1),\ldots,(i_{l(k)},j_{l(k)}),j_{l(k)+1},\ldots,j_k\big)\in\Omega^*$. We define
\begin{eqnarray}\label{flat}
&&\sigma^\flat:=\big((i_1,j_1),\ldots,(i_{\ell(k)},j_{\ell(k)}),j_{\ell(k)+1},\ldots,j_{k-1}\big),\;\;\;\;{\rm if}\;\ell(k)=\ell(k-1);\\
&&\sigma^\flat:=\big((i_1,j_1),\ldots,(i_{\ell(k)-1},j_{\ell(k)-1}),j_{\ell(k)},\ldots,j_{k-1}\big),\;{\rm if}\;\ell(k)=\ell(k-1)+1.\nonumber
\end{eqnarray}
We will consider the following approximate square $F_\sigma$ (cf. \cite{Bed:84,Mcmullen:84}):
\[
F_\sigma:=\bigg[\sum_{h=1}^{\ell(k)}\frac{i_h}{n^h},\sum_{h=1}^{\ell(k)}\frac{i_h}{n^h}+\frac{1}{n^{\ell(k)}}\bigg]\times
\bigg[\sum_{h=1}^k\frac{j_h}{m^h},\sum_{h=1}^k\frac{j_h}{m^h}+\frac{1}{m^k}\bigg].
\]
As is noted in \cite{KZ:16}, we have the following facts:
\begin{eqnarray}\label{diameter}
&&\sqrt{2}m^{-k}\leq|F_\sigma|\leq\sqrt{n^2+1}m^{-k};\\
&&\mu_\sigma:=\mu(F_\sigma)=\prod_{h=1}^{\ell(k)}p_{i_hj_h}\prod_{h=\ell(k)+1}^kq_{j_h}.\nonumber
\end{eqnarray}

We write $\sigma\prec\tau$ and call $\tau$ a descendant of $\sigma$, if $\sigma,\tau\in\Omega^*$ and $F_\tau\subset F_\sigma$.
We say that $\sigma,\tau\in\Omega^*$ are comparable if $\sigma\prec\tau$ or $\tau\prec\sigma$; otherwise, we call them incomparable. We define
\begin{eqnarray}
&&\underline{p}:=\min_{(i,j)\in G}p_{ij},\;\overline{p}:=\max_{(i,j)\in G}p_{ij},\;
\underline{q}:=\min_{j\in G_y}q_j,\;\overline{q}:=\max_{j\in G_y}q_j;\nonumber\\
&&\eta:=\underline{p}\;\underline{q},\;\Lambda_k:=\{\sigma\in\Omega^*:\mu_{\sigma^\flat}\geq\eta^k>\mu_\sigma\},\;k\geq k_0;\label{lambdak}\\
&&\phi_k:={\rm card}(\Lambda_k);\;\underline{\xi}(k):=\min_{\sigma\in\Lambda_k}|\sigma|,\;\overline{\xi}(k):=\max_{\sigma\in\Lambda_k}|\sigma|;\label{xik}.
\end{eqnarray}
\begin{remark}\label{rem0}
(r1) We have $E\subset\bigcup_{\sigma\in\Lambda_k}F_\sigma$ and $\eta\leq\mu_\sigma/\mu_{\sigma^\flat}\leq\overline{q}$ for every $\sigma\in\Omega^*$.
(r2) For $\sigma,\tau\in\Lambda_k$ with $\sigma\neq\tau$, we have
$\sigma\nprec\tau,\;\tau\nprec\sigma$ and  $F_\sigma^\circ\cap F_\tau^\circ=\emptyset$.
\end{remark}

 Let $(a_k)_{k=1}^\infty$ and $(b_k)_{k=1}^\infty$ be number sequences. We write $a_k\lesssim b_k$ if there exists some constant $T$ such that $a_k\leq Tb_k$ for all $k\geq 1$. If $a_k\lesssim b_k$ and $b_k\lesssim a_k$, then we write $a_k\asymp b_k$. By (\ref{lambdak}), one can easily see
\begin{eqnarray}
\underline{p}^{\underline{\xi}(k)}<\eta^k\leq\overline{q}^{\overline{\xi}(k)-1};\label{temp01}
\phi_k\eta^{k+1}\leq\sum_{\sigma\in\Lambda_k}\mu_\sigma=1<\phi_k\eta^k.
 \end{eqnarray}
 Thus, as an immediate consequence of (\ref{temp01}), we obtain
 \begin{eqnarray}
 \underline{\xi}(k),\;\overline{\xi}(k),\;\log \phi_k\asymp k;\;\;\phi_k\leq\phi_{k+1}\leq\eta^{-2}\phi_k.\label{xikcomparable}
 \end{eqnarray}
 The following lemma will allow us to focus on the sequence $(\phi_k)_{k=1}^{\infty}$.
\begin{lemma}\label{sz1}
We define $R_k^{s_0}(\mu):=s_0^{-1}\log\phi_k+\hat{e}_{\phi_k}(\mu)$ and
 \begin{eqnarray}
\underline{P}^{s_0}_0(\mu):=\liminf_{k\to\infty}R_k^{s_0}(\mu);\;
\overline{P}^{s_0}_0(\mu):=\limsup_{k\to\infty}R_k^{s_0}(\mu)\nonumber.
\end{eqnarray}
 Then $\underline{Q}_0^{s_0}(\mu)>0$ iff $\underline{P}^{s_0}_0(\mu)>-\infty$; $\overline{Q}_0^{s_0}(\mu)<+\infty$ iff $\overline{P}^{s_0}_0(\mu)<+\infty$.
\end{lemma}
\begin{proof}
For $l\geq \phi_1$, there exists an $l\in\mathbb{N}$ such that $\phi_k\leq l<\phi_{k+1}$. By (\ref{xikcomparable}) and Theorem 2.5 of \cite{GL:04}, one can easily get the following estimates:
\begin{eqnarray*}
\frac{1}{s_0}\log l+\hat{e}_l(\mu)\left\{ \begin{array}{ll}
\leq\frac{1}{s_0}\log\phi_{k+1}+\hat{e}_{\phi_k}(\mu)\leq R_k^{s_0}(\mu)+\frac{1}{s_0}\log\eta^{-2}\\
\geq\frac{1}{s_0}\log\phi_k+\hat{e}_{\phi_{k+1}}(\mu)\geq Q_{k+1}^{s_0}(\mu)-\frac{1}{s_0}\log\eta^{-2}
\end{array}\right..
\end{eqnarray*}
This, along with (\ref{coefficient}) and (\ref{logen0}), completes the proof of the lemma.
\end{proof}

For every $\sigma\in\Omega^*$, let $h_\sigma$ be an arbitrary similitude on $\mathbb{R}^2$ of similarity ratio $m^{-|\sigma|}$. Define
\[
\nu_\sigma:=\mu(\cdot|F_\sigma)\circ h_\sigma,\;\;K(\sigma):={\rm supp}(\nu_\sigma).
\]
Let $\alpha\subset\mathbb{R}^2$ be a finite set with ${\rm card}(\alpha)=l$, we have
\begin{eqnarray}\label{auxiliary}
\int_{F_\sigma}\log d(x,\alpha)d\mu(x)&=&\mu_\sigma\int \log d(h_\sigma(x),\alpha)d\nu_\sigma(x)\nonumber
\\&\geq&\mu_\sigma(\log m^{-|\sigma|}+\hat{e}_l(\nu_\sigma)).
\end{eqnarray}
\begin{lemma}\label{prez1}
There exist constants $C,t>0$ such that for all $\sigma\in\Omega^*$ and all $\epsilon>0$, the following holds:
\begin{equation}\label{z3}
\sup_{x\in\mathbb{R}^2}\nu_\sigma(B(x,\epsilon))\leq C\epsilon^t.
\end{equation}
\end{lemma}
\begin{proof}
Let $\sigma$ be an arbitrary word in $\Omega^*$. Let $\epsilon_0=\sqrt{n^2+1}m^{-1}$ and $\epsilon\in(0,\epsilon_0)$. There exists a unique $k\geq 1$ such that
\begin{equation}\label{z2}
m^{-(k+1)}\sqrt{n^2+1}\leq\epsilon<\sqrt{n^2+1}m^{-k}.
\end{equation}
 Let $x\in K(\sigma)$. There exists a word $\tau\in\Omega_{|\sigma|+k}$ such that $\sigma\prec\tau$ and $x\in h_\sigma^{-1}(F_\tau)$. We write
\[
\mathcal{A}_\sigma:=\{\tau\in\Omega_{|\sigma|+k}:\sigma\prec\tau,h_\sigma^{-1}(F_\tau)\cap B(x,\epsilon)\neq\emptyset\}.
\]
By \cite[Lemma 9.2]{Fal:04}, one can see that ${\rm card}(\mathcal{A}_\sigma)\leq D_0:=4\pi(n^2+1)$. Using this, (\ref{z2}) and Remark \ref{rem0} (r1), we deduce
\begin{eqnarray*}
\nu_\sigma(B(x,\epsilon))\leq D_0\overline{q}^k\leq
D_0\overline{q}^{\frac{\log\sqrt{n^2+1}-\log m}{\log m}}\epsilon^{-\frac{\log\overline{q}}{\log m}}=:D\epsilon^t.
\end{eqnarray*}
Thus, by \cite[Lemma 12.3]{GL:00}, (\ref{z3}) is fulfilled for $C:=2^t\max\{D,\epsilon_0^{-t}\}$.
\end{proof}

Next, we give estimates for the number of optimal points in the pairwise disjoint neighborhoods of $F_\sigma,\sigma\in\Lambda_k$. Let $\delta:=(n^2+1)^{-1/2}$. We write
\begin{equation}\label{alphasigma}
\alpha_\sigma:=\alpha\cap(F_\sigma)_{\frac{\delta}{8}|F_\sigma|},\;\sigma\in\Lambda_k.
\end{equation}

\begin{remark}\label{rem2}
Let $A_1:=[16\delta^{-1}+5]^2,A_2:=[16\delta^{-1}+3]^2$. By estimating the volumes, for every $\sigma\in\Lambda_k$, one can see the following facts:
\begin{enumerate}
\item[(r3)] $(F_\sigma)_{\frac{\delta}{4}|F_\sigma|}$ can be covered by $A_1$ closed balls of radii $\frac{\delta}{16}|F_\sigma|$ which are centered in $(F_\sigma)_{\frac{\delta}{4}|F_\sigma|}$; let $\beta_\sigma$ be the set of the centers of such $A_1$ balls.
\item[(r4)]$F_\sigma$ can be covered by $A_2$ closed balls of radii $\frac{\delta}{16}|F_\sigma|$ which are centered in $F_\sigma$; let $\gamma_\sigma$ be the set of the centers of such $A_2$ balls.
\item[(r5)]By (\ref{ssc}), one can see that $\alpha_\sigma\cap\alpha_\rho=\emptyset$ for distinct words $\sigma,\rho$ in $\Lambda_k$.
\end{enumerate}
\end{remark}
\begin{lemma}\label{pre} Let $h\in\mathbb{N}$ and $\alpha\in C_{\phi_k}(\mu)$ be given.  Then
\begin{enumerate}
\item[(1)] for $B_h:=t^{-1}(\log h+C)$, we have $\inf_{\sigma\in\Omega^*}\hat{e}_h(\nu_\sigma)\geq B_h$;
\item[(2)]there exists an integer $D_h\geq 1$ such that $\hat{e}_{l-h}(\nu_\sigma)-\hat{e}_l(\nu_\sigma)<\eta\log 2$ for all $l\geq D_h$ and every $\sigma\in\Omega^*$.
\item[(3)]There exists a constant $L_1$ such that
${\rm card}(\alpha_\sigma)\leq L_1$ for all $\sigma\in\Lambda_k$.
\end{enumerate}
\end{lemma}
\begin{proof}
(1) This can be seen by (\ref{z3}) and the proof of Theorem 3.4 in \cite{GL:04}.

(2) This is a consequence of Lemma \ref{prez1} and \cite[Lemma 5.9]{GL:04}. One can see the proof of \cite[Lemma 2.3]{zhu:13} for the argument.

(3) For $\sigma\in\Lambda_k$, let $\beta_\sigma,\gamma_\sigma$ be as defined in Remark \ref{rem2}.
 Let $h_0:=A_1+A_2$ and $L_1:=D_{h_0}$. Suppose that $L_\sigma:={\rm card}(\alpha_\sigma)>L_1$ for some $\sigma$. Then, by Remark \ref{rem2} (r5), there exists a $\tau\in\Lambda_k$ such that $\alpha_\tau=\emptyset$. Let $\widetilde{\alpha}_\sigma\in C_{L_\sigma-h_0}(\nu_\sigma)$. We define
$\beta:=\big(\alpha\setminus(F_\sigma)_{\frac{\delta}{8}|F_\sigma|}\big)\cup\widetilde{\alpha}_\sigma\cup\beta_\sigma\cup\gamma_\tau$.
Then using (2) and (\ref{auxiliary}), one can easily deduce
$\int\log d(x,\beta)d\mu(x)<\int\log d(x,\alpha)d\mu(x)$.
This contradicts the optimality of $\alpha$.
\end{proof}

For two number sequences $(a_k)_{k=1}^\infty$ and $(b_k)_{k=1}^\infty$, we write $a_k\approx b_k$ if there exists a constant $T$ such that $|a_k-b_k|\leq T$ for all (large) $k\in\mathbb{N}$. Now we are able to obtain the following estimate for $\hat{e}_{\phi_k}(\mu)$:
\begin{proposition}\label{pre01}
We have $\hat{e}_{\phi_k}(\mu)\approx \sum_{\sigma\in\Lambda_k}\mu_\sigma\log m^{-|\sigma|}$.
\end{proposition}
\begin{proof}
For $\sigma\in\Lambda_k$, let $b_\sigma$ be an arbitrary point of $F_\sigma$ and set $\beta:=\{b_\sigma\}_{\sigma\in\Lambda_k}$. Then we have that ${\rm card}(\beta)\leq\phi_k$ and $d(x,\beta)\leq d(x,b_\sigma)$ for all $x\in F_\sigma$. Thus,
\begin{equation}\label{s1}
\hat{e}_{\phi_k}(\mu)\leq\sum_{\sigma\in\Lambda_k}\int_{F_\sigma}\log d(x,b_\sigma)d\mu(x)\leq\sum_{\sigma\in\Lambda_k}\mu_\sigma\log|F_\sigma|.
\end{equation}
 Let $\alpha\in C_{\phi_k}(\mu)$. For $\sigma\in\Lambda_k$, let $\gamma_\sigma,\alpha_\sigma$ be as given in Remark \ref{rem2} and (\ref{alphasigma}). By Lemma \ref{pre},  ${\rm card}(\alpha_\sigma\cup\gamma_\sigma)\leq L_1+A_2=:L_2$. Further, for all $x\in F_\sigma$, we have, $d(x,\alpha)\geq d(x,\alpha_\sigma\cup\gamma_\sigma)$. This, (\ref{auxiliary}) and Lemma \ref{pre} (1), yield
\begin{eqnarray*}
\hat{e}_{\phi_k}(\mu)&\geq&\sum_{\sigma\in\Lambda_k}\int_{F_\sigma}\log d(x,\alpha_\sigma\cup\gamma_\sigma)d\mu(x)\\
&=&\sum_{\sigma\in\Lambda_k}\mu_\sigma\int \log d(x,\alpha_\sigma\cup\gamma_\sigma)d\nu_\sigma\circ h_\sigma^{-1}(x)\\
\\&\geq&\sum_{\sigma\in\Lambda_k}\mu_\sigma\log m^{-|\sigma|}+B_{L_2}.
\end{eqnarray*}
This, (\ref{diameter}) and (\ref{s1}) complete the proof of the lemma.
\end{proof}
\begin{remark}
$\hat{e}_{\phi_k}(\mu)$ is closely connected with the sequence $(s_k)_{k=1}^\infty$:
\begin{equation}\label{sk}
s_k:=\bigg(\sum_{\sigma\in\Lambda_k}\mu_\sigma\log m^{-|\sigma|}\bigg)^{-1}\sum_{\sigma\in\Lambda_k}\mu_\sigma\log\mu_\sigma.
\end{equation}
In fact, by Proposition \ref{pre01}  and (\ref{xikcomparable}), one can see that $\hat{e}_{\phi_k}(\mu)\approx s_k^{-1}\log\phi_k$. Thus, the asymptotics of $e_{\phi_k}(\mu)$ reduce to those of the number sequence $(s_k)_{k=1}^\infty$.
\end{remark}

\section{Product coding space and related number sequences}

\subsection{Product coding space}
 Let $G$ and $G_y$ be endowed with discrete topology; let $G^{\mathbb{N}}, G_y^{\mathbb{N}}$ and $W=G^{\mathbb{N}}\times G_y^{\mathbb{N}}$  be endowed with the corresponding product topology. We denote the empty word by $\theta$. Write
\[
G^*:=\bigcup_{k=1}^\infty G^k,\;\;G_y^*:=\bigcup_{k=1}^\infty G_y^k.
\]

Let $\omega=((i_1,j_1),(i_2,j_2),\dots,(i_k,j_k))\in G^k$. For $1\leq h\leq k$, we define
\[
|\omega|:=k,\;\;\omega|_h:=((i_1,j_1),(i_2,j_2),\dots,(i_h,j_h));\;p_\omega:=p_{i_1j_1}\cdots p_{i_kj_k}.
\]
For words $\omega\in G^{\mathbb{N}}$ and $\tau\in G_y^*\cup G_y^{\mathbb{N}}$, we define $|\omega|,\omega|_h,p_\omega$ and $|\tau|,\tau|_h,q_\tau$ in the same manner. In particular,
\[
|\omega|=|\tau|:=\infty,\;p_\omega=q_\tau:=0\;{\rm for}\;\;\omega\in G^{\mathbb{N}},\tau\in G_y^{\mathbb{N}};\;\;p_\theta=q_\theta:=1.
\]
We write $\omega^-:=\omega|_{|\omega|-1}$ if $\omega\in G^*\cup G_y^*$ and $|\omega|>1$; otherwise, let $\omega^-:=\theta$.

For every pair $\omega^{(1)},\omega^{(2)}\in G^*\cup G^{\mathbb{N}}$ or, $\omega^{(1)},\omega^{(2)}\in G_y^*\cup G_y^{\mathbb{N}}$, we write
\[
\omega^{(1)}\prec\omega^{(2)}\;\;{\rm if}\;\;|\omega^{1}|\leq |\omega^{2}|\;\;{\rm and}\;\;\omega^{(1)}=\omega^{(2)}|_{|\omega^{1}|}.
\]

Let $\omega=((i_1,j_1),(i_2,j_2),\dots,(i_k,j_k))\in G^*$ and $\rho=(j_1,\ldots,j_k)\in G_y^*$. We will need to consider the following subsets of $W$:
\begin{eqnarray*}
&&[\omega]:=\{\tau\in G^{\mathbb{N}}:\tau|_k=\omega\};\;\;[\rho]:=\{\tau\in G_y^{\mathbb{N}}:\tau|_k=\rho\};\\
&&[\omega\times\rho]:=[\omega]\times[\rho],\;\;|\omega\times\rho|:=|\omega|+|\rho|.
\end{eqnarray*}
By Kolmogrov consistency theorem, there exists a unique Borel probability measure $\lambda$ on $W$ such that
\begin{equation}\label{lambda}
\lambda([\omega\times\rho])=p_\omega q_\rho.
\end{equation}
\begin{remark}
Compared with the measure $\mu$, the advantage of $\lambda$ lies in the fact that it possesses a kind of independence. This will help us to estimate the geometric mean error in a convenient manner.
\end{remark}
For our purpose, we will focus on the following set:
\begin{eqnarray*}
\Phi^*:=\{\omega\times\rho:\omega\in G^*,\;\rho\in G_y^*;\;\;|\omega|=[(|\omega\times\rho|)\vartheta]\}.
\end{eqnarray*}

For two words $\sigma^{(i)}=\omega^{(i)}\times\rho^{(i)}\in \Phi^*, i=1,2$, we write $\sigma^{(1)}\prec\sigma^{(2)}$ and call $\sigma^{(2)}$ a descendant of $\sigma^{(1)}$ if $\omega^{(1)}\prec\omega^{(2)}$ and  $\rho^{(1)}\prec\rho^{(2)}$.
We say that $\sigma^{(1)},\sigma^{(2)}$ are comparable if $\sigma^{(1)}\prec\sigma^{(2)}$ or $\sigma^{(2)}\prec\sigma^{(1)}$; otherwise we call them incomparable. We write $\sigma^{(1)}=(\sigma^{(2)})^-$ if
$\sigma^{(1)}\prec\sigma^{(2)}\;\;{\rm and}\;\;\sigma^{(2)}|=|\sigma^{(1)}|+1$.

Let $\omega=((i_1,j_1),(i_2,j_2),\dots,(i_{\ell(k)},j_{\ell(k)}))\in G^*$ and $\rho=(j_{{\ell(k)}+1},\ldots,j_k)\in G_y^*$.  Then for $\sigma=\omega\times\rho$, we have
\begin{eqnarray}\label{predicessor}
\sigma^-=:\left\{ \begin{array}{ll}
\omega\times\rho^-\;\;&\mbox{if}\;\;\ell(k)=\ell(k-1)\\
\omega^-\times\rho\;\;&\mbox{if}\;\;\ell(k)=\ell(k-1)+1
\end{array}\right..
\end{eqnarray}
\begin{remark}\label{rem3}
Let $\sigma^{(i)}=\omega^{(i)}\times\rho^{(i)}\in \Phi^*, i=1,2$. Then for $[\sigma^{(i)}],i=1,2$ , we have, either they are disjoint, or, one is a subset of the other. This can be seen as follows. If either $\omega^{(i)},i=1,2$, or $\rho^{(i)},i=1,2$, are incomparable, then clearly $[\sigma^{(i)}],i=1,2$, are disjoint; if both $\omega^{(i)},i=1,2$, and $\rho^{(i)},i=1,2$, are comparable, then one of $[\sigma^{(1)}],[\sigma^{(2)}]$, is contained in the other, since, by the definition of $\Phi^*$, $|\omega^{(1)}|>|\omega^{(2)}|$ implies that $|\rho^{(1)}|\geq|\rho^{(2)}|$.
\end{remark}

A finite subset $\Upsilon$ of $\Phi^*$ is called a finite maximal antichain in $W$ if the words in $\Upsilon$ are pairwise incomparable and $W=\bigcup_{\tau\in\Upsilon}[\tau]$.

\subsection{Two related number sequences}
For the proof of the main result, we need to study the asymptotic order of two number sequences which are related to $(s_k)_{k=1}^\infty$ (see (\ref{sk})).
The first sequence is about the words in $\Phi^*$ of the same length. We define
\begin{equation}\label{uk}
\Phi_k:=\{\widetilde{\sigma}\in \Phi^*:|\widetilde{\sigma}|=k\};\;U_k:=\sum_{\widetilde{\sigma}\in \Phi_k}\lambda([\widetilde{\sigma}])\log \lambda([\widetilde{\sigma}]);\;d_k:=\frac{U_k}{-k\log m}.
\end{equation}
\begin{lemma}\label{lem001}
We have $|d_k-s_0|\lesssim k^{-1}$.
\end{lemma}
\begin{proof}
By the definition of $d_k$, we have
\begin{eqnarray}\label{sgz2}
d_k&=&\frac{1}{-k\log m}\bigg(\ell(k)\sum_{(i,j)\in G}p_{ij}\log p_{ij}+(k-\ell(k))\sum_{j\in G_y}q_j\log q_j\bigg)\nonumber\\
&=&\frac{1}{\log m}\bigg(\frac{\ell(k)}{k}\sum_{j\in G_y}\sum_{i\in G_x(j)}p_{ij}\log p_{ij}^{-1}+\frac{k-\ell(k)}{k}\sum_{j\in G_y}q_j\log q_j^{-1}\bigg).
\end{eqnarray}
Note that $p_{ij}\leq q_j$ and $q_j=\sum_{i\in G_x(j)}p_{ij}$ for all $j\in G_y$. We obtain
\begin{equation}\label{sgz3}
\sum_{j\in G_y}\sum_{i\in G_x(j)}p_{ij}\log p_{ij}^{-1}\geq\sum_{j\in G_y}q_j\log q_j^{-1}.
\end{equation}
We consider the function $g$ as defined below:
\[
 g: x\mapsto x\sum_{j\in G_y}\sum_{i\in G_x(j)}p_{ij}\log p_{ij}^{-1}+(1-x)\sum_{j\in G_y}q_j\log q_j^{-1}.
 \]
 Using (\ref{sgz3}), one can see that $g$ is increasing. Since $\ell(k)\leq k\vartheta$, we deduce that $d_k\leq s_0$. Note that $|k\vartheta-\ell(k)|\leq 1$. This, along with (\ref{sgz2}) and (\ref{sgz3}), yields
\begin{eqnarray}\label{sgz1}
0\leq s_0-d_k\leq \frac{2}{k\log m}\sum_{(i,j)\in G}p_{ij}\log p_{ij}^{-1}\lesssim k^{-1}.
\end{eqnarray}
This completes the proof of the lemma.
\end{proof}

The second sequence is related to the words in $\Phi^*$ which are typically of different length.
Let $(\Gamma_k)_{k=1}^\infty$ be a sequence of finite maximal antichains in $W$. We define
 \begin{eqnarray}\label{temp03}
 &&t_k:=t(\Gamma_k):=\frac{\sum_{\widetilde{\sigma}\in\Gamma_k}\lambda(\widetilde{\sigma})\log \lambda(\widetilde{\sigma})}{\sum_{\widetilde{\sigma}\in\Gamma_k}\lambda(\widetilde{\sigma})\log m^{-|\widetilde{\sigma}|}};\nonumber\\&&\underline{l}_k:=l(\Gamma_k):=\min_{\widetilde{\sigma}\in\Gamma_k}|\widetilde{\sigma}|,
 \;\overline{l}_k:=\max_{\widetilde{\sigma}\in\Gamma_k}|\widetilde{\sigma}|.
 \end{eqnarray}

\begin{lemma}\label{lem01}
We have $|t_k-s_0|\lesssim \underline{l}_k^{-1}$.
\end{lemma}
\begin{proof}
For every $h\geq k_0$, let $U_h$ and $d_h$ be as defined in (\ref{uk}). Then for $\sigma\in\Gamma_k$,
\begin{eqnarray*}
U_{|\sigma|}=\ell(|\sigma|)\sum_{(i,j)\in G}p_{ij}\log p_{ij}+(|\sigma|-\ell(|\sigma|))\sum_{j\in G_y}q_j\log q_j.
\end{eqnarray*}
For every integer $h\geq 0$ and $\sigma=\omega\times\rho\in\Gamma_k$, we define
\begin{eqnarray*}
\zeta(\sigma):=\lambda([\sigma])(\log\lambda([\sigma])-U_{|\sigma|});\;
\Upsilon(\sigma,h):=\{\rho\in\Phi_{|\sigma|+h}:\sigma\prec\rho\}.
\end{eqnarray*}
If $\ell(|\sigma|+1)=\ell(|\sigma|)$, then $U_{|\sigma|+1}-U_{|\sigma|}=\sum_{j\in G_y}q_j\log q_j$. We deduce
\begin{eqnarray*}
\sum_{\rho\in\Upsilon(\sigma,1)}\zeta(\rho)&=&\sum_{j\in G_y}\lambda([\omega\times(\rho\ast j)])\big(\log\lambda([\omega\times(\rho\ast j)])-U_{|\sigma|+1}\big)\\
&=&\lambda([\sigma])\big(\log\lambda([\sigma])+\sum_{j\in G_y}q_j\log q_j\big)-\lambda([\sigma])U_{|\sigma|+1}\\&=&
\lambda([\sigma])(\log\lambda([\sigma])-U_{|\sigma|})=\zeta(\sigma).
\end{eqnarray*}
If $\ell(|\sigma|+1)=\ell(|\sigma|)+1$, then $U_{|\sigma|+1}-U_{|\sigma|}=\sum_{(i,j)\in G}p_{ij}\log p_{ij}$. We have
\begin{eqnarray*}
\sum_{\rho\in\Upsilon(\sigma,1)}\zeta(\rho)&=&\sum_{(i,j)\in G}\lambda([(\omega\ast (i,j))\times\rho])(\log(\lambda([(\omega\ast (i,j))\times\rho])-U_{|\sigma|+1})\\
&=&\lambda([\sigma])\big(\log\lambda([\sigma])+\sum_{(i,j)\in G}p_{ij}\log p_{ij}\big)-\lambda([\sigma])U_{|\sigma|+1}\\&=&
\lambda([\sigma])(\log\lambda([\sigma])-U_{|\sigma|})=\zeta(\sigma).
\end{eqnarray*}
By induction, for all $h\geq 1$ and all $\sigma\in\Gamma_k$, we have, $\sum_{\rho\in\Upsilon(\sigma,h)}\zeta(\rho)=\zeta(\sigma)$. Using this fact, we deduce
\begin{eqnarray*}
\sum_{\sigma\in\Gamma_k}\zeta(\sigma)=\sum_{\sigma\in \Phi_{\underline{l}_k}}\zeta(\sigma)=0.
\end{eqnarray*}
This, along with the definition of $\zeta(\sigma)$, yields
\[
\sum_{\sigma\in\Gamma_k}\lambda([\sigma])\log\lambda([\sigma])=\sum_{\sigma\in\Gamma_k}\lambda([\sigma])U_{|\sigma|}.
\]
Thus, we obtain
\[
\min_{\underline{l}_k\leq h\leq \overline{l}_k}d_h\leq t_k=\frac{\sum_{\sigma\in\Gamma_k}\lambda([\sigma])U_{|\sigma|}}{\sum_{\sigma\in\Gamma_k}\lambda([\sigma])\log m^{-|\sigma|}}\leq\max_{\underline{l}_k\leq h\leq \overline{l}_k}d_h.
\]
Thus, the lemma follows by the preceding inequality and Lemma \ref{lem001}.
\end{proof}

\section{Proof of Theorem \ref{mthm} }

Let $\Omega^*$ be as defined in section 2. We need to associate words in $\Omega^*$ with words in $\Phi^*$. For $\sigma=((i_1,j_1),\ldots,(i_{\ell(k)},j_{\ell(k)}),j_{\ell(k)+1},\ldots,j_k)\in\Omega^*$, we define
\[
\mathcal{L}(\sigma):=((i_1,j_1),\ldots,(i_{\ell(k)},j_{\ell(k)}))\times (j_{\ell(k)+1},\ldots,j_k).
\]
Then $\Phi_k:=\{\mathcal{L}(\sigma):\sigma\in\Omega_k\}$ and $\Phi^*=\bigcup_{k\geq 1}\Phi_k$. By (\ref{lambda}), we have
\begin{eqnarray}\label{equivalent}
\lambda([\mathcal{L}(\sigma)])=\mu_\sigma,\;|\mathcal{L}(\sigma)|=|\sigma|;\;\sigma\in\Omega^*.
\end{eqnarray}
\begin{remark}
The difference between $\Omega^*$ and $\Phi^*$ lies in the fact that they have different partial orders. The partial order on $\Omega^*$ is defined according to the geometric construction of the carpet $E$, but this is not so for $\Phi^*$. As a consequence, the words in $\Omega^*$ and those in $\Phi^*$ have descendants in different ways (cf. (\ref{flat}) and (\ref{predicessor})).
\end{remark}
\begin{remark}\label{rem1}
We write $\widetilde{\Lambda}_k:=\{\mathcal{L}(\sigma):\sigma\in\Lambda_k\}$. Let $\sigma^{(i)}\in\Lambda_k, i=1,2$, be distinct words. We know that $F^\circ_{\sigma^{(1)}}\cap F^\circ_{\sigma^{(2)}}=\emptyset$. However, it may happen that the sets $[\mathcal{L}(\sigma^{(i)})]\in\widetilde{\Lambda}_k, i=1,2$, are overlapping. This can be seen as follows. It is possible that both the following words belong to $\Lambda_k$:
\begin{eqnarray*}
&&\sigma^{(1)}=((i_1,j_1),\ldots,(i_{\ell(k)},j_{\ell(k)}),j_{\ell(k)+1},\ldots,j_k);\\
&&\sigma^{(2)}=((i_1,j_1),\ldots,(i_{\ell(k)},j_{\ell(k)}),(i,j),j_{\ell(k)+1},\ldots,j_k).
\end{eqnarray*}
We clearly have that $\mathcal{L}(\sigma^{(1)})\prec \mathcal{L}(\sigma^{(2)})$ and $[\mathcal{L}(\sigma^{(2)})]\subset[\mathcal{L}(\sigma^{(1)})]$.
\end{remark}

The possible overlappings as described in Remark \ref{rem1} prevents us from further estimates of the geometric mean error; and as we mentioned in section 1, the method in \cite{zhu:18} is no longer applicable. In order to remove the possible overlappings in $\widetilde{\Lambda}_k$, we are going to replace $\widetilde{\Lambda}_k$ with some maximal finite antichain in $W$. For this purpose, we need a finite sequence of integers which will be defined as follows.
Let $\overline{\xi}(k)$ and $\underline{\xi}(k)$ be as defined in (\ref{xik}). Then we have
\[
\max_{\widetilde{\sigma}\in\widetilde{\Lambda}_k}|\widetilde{\sigma}|=\overline{\xi}(k),\;\min_{\widetilde{\sigma}\in\widetilde{\Lambda}_k}|\widetilde{\sigma}|=\underline{\xi}(k).
\]
We write
$M:=\ell(\overline{\xi}(k))-\ell(\underline{\xi}(k))+1$. Let $\xi_1(k):=\underline{\xi}(k)$. We define
\[
\xi_2(k)=\min\{\xi_1(k)<h\leq\overline{\xi}(k): \ell(h)=\ell(\xi_1(k))+1\}.
\]
Assume that $\xi_j(k)$ is defined. We then define
\[
\xi_{j+1}(k)=\min\{\xi_j(k)<h\leq\overline{\xi}(k): \ell(h)=\ell(\xi_j(k))+1\}.
\]
By induction, the sequence $\big(\xi_j(k)\big)_{j=1}^M$ is well defined.

Next, we construct a finite maximal antichain $\widetilde{\Lambda}_M(k)$ in $W$ by induction.

Let $\widetilde{\Lambda}_1(k):=\widetilde{\Lambda}_k$ and $\mathcal{F}_1(k)=\mathcal{G}_1(k):=\emptyset$. We will construct two sets $\mathcal{F}_2(k)\subset\widetilde{\Lambda}_1(k)\cap\Phi_{\xi_2(k)}$ and $\mathcal{G}_2(k)\subset\Phi_{\xi_2(k)}\setminus\widetilde{\Lambda}_1(k)$ and the define
\[
\widetilde{\Lambda}_2(k):=(\widetilde{\Lambda}_1(k)\setminus\mathcal{F}_2(k))\cup\mathcal{G}_2(k),
 \]
 so that those words in $\widetilde{\Lambda}_2(k)$ with length not exceeding $\xi_2(k)$, are pairwise incomparable. The following Lemmas \ref{z1}-\ref{lem3} are devoted to this goal.
\begin{lemma}\label{z1}
Assume that $\xi_2(k)>\xi_1(k)+1$. Then the words in the following set are pairwise incomparable:
$\Gamma_1(k):=\bigcup_{h=\xi_1(k)}^{\xi_2(k)-1}\big(\widetilde{\Lambda}_1(k)\cap\Phi_h\big)$.
\end{lemma}
\begin{proof}
Let $\widetilde{\sigma},\widetilde{\rho}\in\Gamma_1(k)$ with $\widetilde{\sigma}\neq\widetilde{\rho}, \sigma=\mathcal{L}^{-1}(\widetilde{\sigma}),\rho=\mathcal{L}^{-1}(\widetilde{\rho})$. By (\ref{lambdak}) and (\ref{lambda}),
\begin{equation}\label{s2}
\lambda([\widetilde{\sigma}])=\mu(F_\sigma)<\eta^k,\;\lambda([\widetilde{\rho}])=\mu(F_\rho)<\eta^k.
\end{equation}
If $|\widetilde{\sigma}|=|\widetilde{\rho}|$, then $\widetilde{\sigma},\widetilde{\rho}$ certainly incomparable. Next, we assume that $|\widetilde{\sigma}|<|\widetilde{\rho}|$.
Then $\xi_1(k)<|\rho|<\xi_2(k)$ and $\ell(|\rho|)=\ell(\xi_1(k))=\ell(|\rho|-1)$. Thus,  the word $\widetilde{\rho}^-$ takes the first form in (\ref{predicessor}); and $\rho^\flat$ take the form in (\ref{flat}). From this, we deduce that $\mathcal{L}^{-1}(\widetilde{\rho}^-)=\rho^\flat$. It follows by using (\ref{s2}) that
\[
\lambda([\widetilde{\rho}^-])=\mu(F_{\rho^\flat})\geq\eta^k>\lambda([\widetilde{\sigma}]).
\]
It follows that $\widetilde{\sigma}\nprec\widetilde{\rho}$. Since $|\widetilde{\sigma}|<|\widetilde{\rho}|$, we conclude that $\widetilde{\sigma},\widetilde{\rho}$ are incomparable.
\end{proof}

Next, we consider the words in $\widetilde{\Lambda}_1(k)\cap\Phi_{\xi_2(k)}$.

\begin{lemma}\label{lem2}
Let $\widetilde{\sigma}\in\widetilde{\Lambda}_1(k)\cap\Phi_{\xi_2(k)}$ with
\begin{equation}\label{s4}
\widetilde{\sigma}=((i_1,j_1),\ldots,(i_{\ell(\xi_2(k))},j_{\ell(\xi_2(k))})\times(j_{\ell(\xi_2(k))+1},\ldots,j_{\xi_2(k)}).
\end{equation}
Assume that $\widetilde{\omega}\prec\widetilde{\sigma}$ for some $\widetilde{\omega}\in\Gamma_1(k)$. Then for every $i\in G_x(j_{\ell(\xi_2(k))})$,
\begin{equation}\label{s5}
\widetilde{\sigma}(i):=((i_1,j_1),\ldots,(i,j_{\ell(\xi_2(k))})\times(j_{\ell(\xi_2(k))+1},\ldots,j_{\xi_2(k)})\in\widetilde{\Lambda}_1(k).
\end{equation}
\end{lemma}
\begin{proof}
Let $\widetilde{\sigma}\in\widetilde{\Lambda}_1(k)\cap\Phi_{\xi_2(k)}$. By (\ref{lambdak}), we have
\[
\mu_{\mathcal{L}^{-1}(\widetilde{\sigma})}<\eta^k\leq\mu_{(\mathcal{L}^{-1}(\widetilde{\sigma}))^\flat}.
\]
By the hypothesis, $\widetilde{\omega}\prec\widetilde{\sigma}$ for some $\widetilde{\omega}\in\Gamma_1(k)$. Since $\ell(|\widetilde{\sigma}|)=\ell(|\widetilde{\sigma}|-1)+1$, the word  $\widetilde{\sigma}^-$ takes the second form in (\ref{predicessor}). By Remark \ref{rem3}, we obtain $\widetilde{\omega}\prec\widetilde{\sigma}^-$. Hence, by the definition of $\Gamma_1(k)$ and (\ref{lambdak}), we obtain
\begin{equation}\label{sz2}
\mu_{\mathcal{L}^{-1}(\widetilde{\sigma}^-)}=\lambda([\widetilde{\sigma}^-])\leq\lambda([\widetilde{\omega}])=\mu_{\mathcal{L}^{-1}(\widetilde{\omega})}<\eta^k.
\end{equation}
Note that $\widetilde{\sigma}(i)^-=\widetilde{\sigma}^-$  and $(\mathcal{L}^{-1}(\widetilde{\sigma}(i)))^\flat=(\mathcal{L}^{-1}(\widetilde{\sigma}))^\flat$ for every $i\in G_x(j_{\xi_2(k)})$. From this and (\ref{sz2}), we deduce
\begin{eqnarray*}
\mu_{\mathcal{L}^{-1}(\widetilde{\sigma}(i))}<\mu_{\mathcal{L}^{-1}(\widetilde{\sigma}(i)^-)}=\mu_{\mathcal{L}^{-1}(\widetilde{\sigma}^-)}<\eta^k
\leq\mu_{(\mathcal{L}^{-1}(\widetilde{\sigma}))^\flat}=\mu_{(\mathcal{L}^{-1}(\widetilde{\sigma}(i)))^\flat}.
\end{eqnarray*}
By (\ref{lambdak}), we know that $\mathcal{L}^{-1}(\widetilde{\sigma}(i))\in\Lambda_k$ and $\widetilde{\sigma}(i)\in\widetilde{\Lambda}_1(k)$.
\end{proof}

We denote the set of all the words $\widetilde{\sigma}$ that fulfills the assumption in Lemma \ref{lem2} by $\mathcal{F}_2(k)$. For every $\widetilde{\sigma}\in\mathcal{F}_2(k)$, let us denote the set of all the words $\widetilde{\sigma}(i)$ as given in (\ref{s5}) by $\mathcal{F}(\widetilde{\sigma})$. Clearly, we have
\[
\mathcal{F}(\widetilde{\sigma}(i))=\mathcal{F}(\widetilde{\sigma})\;\;{\rm for}\;i\in G_x(j_{\ell(\xi_2(k))}).
\]
For every $\widetilde{\sigma}\in\mathcal{F}_2(k)$, we fix an arbitrary $\widetilde{\sigma}(i)$, and denote the set of all these words $\widetilde{\sigma}(i)$ by $\widetilde{\mathcal{F}_2}(k)$. Then we have $\mathcal{F}_2(k)=\bigcup_{\widetilde{\sigma}\in\widetilde{\mathcal{F}_2}(k)}\mathcal{F}(\widetilde{\sigma})$, and for every pair of distinct words $\widetilde{\sigma},\widetilde{\rho}\in\widetilde{\mathcal{F}_2}(k)$, we have $\mathcal{F}(\widetilde{\sigma})\cap \mathcal{F}(\widetilde{\rho})=\emptyset$.
\begin{lemma}\label{lem3}
Let $\widetilde{\sigma}\in\widetilde{\mathcal{F}_2}(k)$ be as given in (\ref{s4}). For every $i\in G_x(j_{\xi_2(k)})$, we define $\widehat{\sigma}(i)$ to be (by interchanging the positions of $j_{\xi_2(k)}$ and $j_{\ell(\xi_2(k))}$ in (\ref{s5})):
\begin{eqnarray*}
\big((i_1,j_1),\ldots,(i_{\ell(\xi_2(k))-1},j_{\ell(\xi_2(k))-1}),(i,j_{\xi_2(k)})\big)\times\big(j_{\ell(\xi_2(k))+1},\ldots,j_{\xi_2(k)-1},j_{\ell(\xi_2(k))}\big)
\end{eqnarray*}
and let $\mathcal{G}(\widetilde{\sigma}):=\{\widehat{\sigma}(i): i\in G_x(j_{\xi_2(k)})\}$.
Then for every $i\in G_x(j_{\xi_2(k)})$, we have
\begin{enumerate}
\item[(a1)] $\widehat{\sigma}(i)\notin\widetilde{\Lambda}_1(k)$; $\lambda([\widehat{\sigma}(i)])<\eta^k$;
\item[(a2)]$\lambda([(\widehat{\sigma}(i))^-])\geq\eta^k$ and $(\widehat{\sigma}(i))^-\notin\widetilde{\Lambda}_1(k)$;
\end{enumerate}
\end{lemma}
\begin{proof}
(a1) Note that $|\widehat{\sigma}(i)|=|\widetilde{\sigma}|=\xi_2(k)$.
By the definition of $\xi_2(k)$, the word $(\mathcal{L}^{-1}(\widehat{\sigma}(i)))^\flat$ takes the following form:
\[
((i_1,j_1),\ldots,(i_{\ell(\xi_2(k))-1},j_{\ell(\xi_2(k))-1},j_{\xi_2(k)},j_{\ell(\xi_2(k))+1},\ldots,j_{\xi_2(k)-1}).
\]
Note that $j_{\xi_2(k)},j_{\ell(\xi_2(k))+1}\ldots,j_{\xi_2(k)-1}$ is a rearrangement of $j_{\ell(\xi_2(k))+1}\ldots,j_{\xi_2(k)}$. We obtain that $\mu_{(\mathcal{L}^{-1}(\widehat{\sigma}(i)))^\flat}=\mu_{\mathcal{L}^{-1}(\widetilde{\sigma}^-)}$. This, (\ref{lambdak}) and (\ref{sz2}) yield
\begin{equation}\label{temp1}
\lambda([\widehat{\sigma}(i)])=\mu_{L^{-1}(\widehat{\sigma}(i))}<\mu_{(\mathcal{L}^{-1}(\widehat{\sigma}(i)))^\flat}=
\mu_{\mathcal{L}^{-1}(\widetilde{\sigma}^-)}\leq\mu_{\mathcal{L}^{-1}(\widetilde{\omega})}<\eta^k.
\end{equation}
By (\ref{lambdak}), one gets that $\mathcal{L}^{-1}(\widehat{\sigma}(i))\notin\Lambda_k$. Hence, $\widehat{\sigma}(i)\notin\widetilde{\Lambda}_1(k)$.

(a2) Since $|\widehat{\sigma}(i)|=\xi_2(k)$. By (\ref{s4}), we know that
\[
\widehat{\sigma}(i)^-=((i_1,j_1),\ldots,(i_{\ell(\xi_2(k))-1},j_{\ell(\xi_2(k))-1})\times(j_{\ell(\xi_2(k))+1},\ldots,j_{\xi_2(k)-1},j_{\ell(\xi_2(k))}).
\]
On the other hand, one easily sees that $(\mathcal{L}^{-1}(\widetilde{\sigma}))^\flat$ takes the following form:
\[
((i_1,j_1),\ldots,(i_{\ell(\xi_2(k))-1},j_{\ell(\xi_2(k))-1}),j_{\ell(\xi_2(k))},j_{\ell(\xi_2(k))+1},\ldots,j_{\xi_2(k)-1}).
\]
Since $\widetilde{\sigma}\in\widetilde{\Lambda}_1(k)$, we have $\mu_{\mathcal{L}^{-1}(\widetilde{\sigma}^\flat)}\geq\eta^k$. Note that $j_{\ell(\xi_2(k))+1},\ldots,j_{\xi_2(k)-1},j_{\ell(\xi_2(k))}$ is a rearrangement of $j_{\ell(\xi_2(k))},j_{\ell(\xi_2(k))+1},\ldots,j_{\xi_2(k)-1}$. By (\ref{lambdak}), we obtain
\begin{equation}\label{tem2}
\lambda([\widehat{\sigma}(i)^-])=\mu_{\mathcal{L}^{-1}(\widehat{\sigma}(i)^-)}=\mu_{(\mathcal{L}^{-1}(\widetilde{\sigma}))^\flat}\geq\eta^k.
\end{equation}
Again, by (\ref{lambdak}), we have that $\widehat{\sigma}(i)^-\notin\widetilde{\Lambda}_1(k)$.
\end{proof}

With the above preparations, we now define
\begin{equation}\label{2k}
\mathcal{G}_2(k):=\bigcup_{\widetilde{\sigma}\in\widetilde{\mathcal{F}}_2(k)}\mathcal{G}(\widetilde{\sigma}),\;\;\widetilde{\Lambda}_2(k):=
\big(\widetilde{\Lambda}_1(k)\setminus\mathcal{F}_2(k)\big)\cup\mathcal{G}_2(k).
\end{equation}

Using the next lemma, we give characterizations for the words in $\widetilde{\Lambda}_2(k)$. This will be useful in the construction of the maximal antichain which is mentioned above.
\begin{lemma}\label{lem6}
Let $\mathcal{G}_2(k)$ and $\widetilde{\Lambda}_2(k)$ be as defined in (\ref{2k}). Then
\begin{enumerate}
\item[(b1)] $\widetilde{\Lambda}_2(k)\setminus\mathcal{G}_2(k)\subset\widetilde{\Lambda}_1(k)$;
\item[(b2)] $\lambda([\widehat{\tau}])<\eta^k$ for all $\widehat{\tau}\in\widetilde{\Lambda}_2(k)$ and $\lambda([\widehat{\tau}^-])\geq\eta^k$ for every $\widehat{\tau}\in\mathcal{G}_2(k)$;

\item[(b3)] the words in $\widetilde{\Lambda}_2(k)\cap\big(\bigcup_{h=\xi_1(k)}^{\xi_2(k)}\Phi_h\big)$ are pairwise incomparable;
\end{enumerate}
\end{lemma}
\begin{proof}
(b1) This is an immediate consequence of (\ref{2k}).

(b2) By (\ref{temp1}) and (\ref{tem2}), we have, $\lambda([\widehat{\tau}])<\eta^k\leq \lambda([\widehat{\tau}^-])$ for every $\widehat{\tau}\in\mathcal{G}_2(k)$. For every $\widehat{\tau}\in\widetilde{\Lambda}_2(k)\setminus\mathcal{G}_2(k)\subset\widetilde{\Lambda}_1(k)$, we certainly have $\lambda([\widehat{\tau}])<\eta^k$.

(32) Let $\widetilde{\sigma}\in\Gamma_1(k)$ (see Lemma \ref{z1}) and $\widehat{\tau}\in\mathcal{G}_2(k)$. Then $|\widetilde{\sigma}|<|\widehat{\tau}|$. By (\ref{lambdak}) and (b1), we have $\lambda([\widetilde{\sigma}])<\eta^k\leq\lambda([\widehat{\tau}^-])$, which implies that $\widehat{\tau}$ and $\widetilde{\sigma}$ are incomparable. By the definition of $\mathcal{F}_2(k)$, for every $\widehat{\tau}\in(\widetilde{\Lambda}_1(k)\cap\Phi_{\xi_2(k)})\setminus\mathcal{F}_2(k)$ and $\widehat{\sigma}\in\Gamma_1(k)$, we have $\widehat{\tau},\widehat{\sigma}$ are incomparable; in addition, such a word $\widehat{\tau}$ is certainly incomparable with every word in $\mathcal{G}_2(k)$ since they are different words and are of the same length. Combining the above analysis and Lemma \ref{z1}, we obtain (b3).
\end{proof}

Let us proceed with the construction of $\widetilde{\Lambda}_M(k)$. Assume that for $l\geq 2$ and $1\leq h\leq l$, the sets
 \[
\mathcal{F}_h(k)\subset\widetilde{\Lambda}_{h-1}(k)\cap\widetilde{\Phi}_{\xi_h(k)},\mathcal{G}_h(k)\subset\widetilde{\Phi}_{\xi_h(k)}\setminus\widetilde{\Lambda}_{h-1}(k)\;
 {\rm and} \;\widetilde{\Lambda}_h(k)
 \]
 are defined such that the following (c1)-(c3) are fulfilled for all $2\leq h\leq l$:
\begin{enumerate}
\item[(c1)]$\widetilde{\Lambda}_h(k):=(\widetilde{\Lambda}_{h-1}(k)\setminus\mathcal{F}_h(k))\cup\mathcal{G}_h(k)=
\big(\widetilde{\Lambda}_1(k)\setminus\bigcup_{p=1}^h\mathcal{F}_p(k)\big)\cup\big(\bigcup_{p=1}^h\mathcal{G}_p(k)\big)$;
\item[(c2)]$\lambda([\widehat{\tau}])<\eta^k$ for all $\widehat{\tau}\in\widetilde{\Lambda}_h(k)$ and $\lambda([\widehat{\tau}^-])\geq\eta^k$ for every $\widehat{\tau}\in\mathcal{G}_h(k)$;
\item[(c3)]the words in $\widetilde{\Lambda}_h(k)\cap\big(\bigcup_{p=\xi_1(k)}^{\xi_h(k)}\Phi_p\big)$ are pairwise incomparable.
\end{enumerate}
Next, we define three sets $\mathcal{F}_{l+1}(k),\mathcal{G}_{l+1}(k)$ and $\widetilde{\Lambda}_{l+1}(k)$ such that (c1)-(c3) hold for $h=l+1$.

\emph{Claim 1}:  The words in the following set are pairwise incomparable:
\[
\Gamma_l(k):=\bigcup_{h=\xi_1(k)}^{\xi_{l+1}(k)-1}\big(\widetilde{\Lambda}_l(k)\cap\Phi_h\big).
\]
This can be seen as follows. If $\xi_{l+1}(k)=\xi_l(k)+1$, then
\[
\Gamma_l(k):=\bigcup_{h=\xi_1(k)}^{\xi_l(k)}\big(\widetilde{\Lambda}_l(k)\cap\Phi_h\big);
\]
and the claim follows from (c3). Next we assume that $\xi_{l+1}(k)>\xi_l(k)+1$.
By (c1), we have
\[
\mathcal{H}_l:=\widetilde{\Lambda}_l(k)\cap\bigg(\bigcup_{h=\xi_l(k)+1}^{\xi_{l+1}(k)-1}\Phi_h\bigg)
=\widetilde{\Lambda}_1(k)\cap\bigg(\bigcup_{h=\xi_l(k)+1}^{\xi_{l+1}(k)-1}\Phi_h\bigg)\subset\widetilde{\Lambda}_1(k).
\]
For every pair of distinct words $\widehat{\rho},\widehat{\tau}\in\mathcal{H}_l$, if $|\widehat{\tau}|=|\widehat{\rho}|$, then they are certainly incomparable; otherwise, we may assume that $|\widehat{\tau}|>|\widehat{\rho}|$. Note that
\[
\ell(|\widehat{\tau}|)=\ell(|\widehat{\tau}|-1)=\ell(\xi_l(k)).
\]
We have $\mathcal{L}^{-1}(\widehat{\tau}^-)=(\mathcal{L}^{-1}(\widehat{\tau}))^\flat$. Hence,
\begin{equation}\label{sz3}
\lambda([\widehat{\tau}^-])=\mu_{(\mathcal{L}^{-1}(\widehat{\tau})^-)}=\mu_{(\mathcal{L}^{-1}(\widehat{\tau}))^\flat}\geq\eta^k>\lambda([\widehat{\rho}]).
\end{equation}
This implies that $\widehat{\rho},\widehat{\tau}$ are incomparable. By (c2), we know that (\ref{sz3}) also holds for $\widehat{\tau}\in\mathcal{H}_l$ and $\widehat{\rho}\in\widetilde{\Lambda}_l(k)\cap\big(\bigcup_{h=\xi_1(k)}^{\xi_l(k)}\Phi_h\big)$. Since $|\widehat{\tau}|\geq\xi_l(k)+1>|\widehat{\rho}|$, we obtain that $\widehat{\rho},\widehat{\tau}$ are incomparable.
Combining the above analysis and (c3), the claim follows.

We denote by $\mathcal{F}_{l+1}(\widetilde{\sigma})$ the set of all the words $\widetilde{\sigma}\in\widetilde{\Lambda}_l(k)\cap\Phi_{\xi_{l+1}(k)}$ such that $\widetilde{\omega}\prec\widetilde{\sigma}$ for some $\widetilde{\omega}\in\Gamma_l(k)$. For every $\widetilde{\sigma}\in\mathcal{F}_{l+1}(k)$, let $\mathcal{F}(\widetilde{\sigma})$ and $\mathcal{G}(\widetilde{\sigma})$ be defined in the same way as we did for $\widetilde{\sigma}\in\mathcal{F}_2(k)$ and let $\widetilde{\mathcal{F}}_{l+1}(k)$ be defined accordingly. We define
\[
\mathcal{G}_{l+1}(k):=\bigcup_{\widetilde{\sigma}\in\widetilde{\mathcal{F}}_{l+1}(k)}\mathcal{G}(\widetilde{\sigma});\;\;
\widetilde{\Lambda}_{l+1}(k):=\big(\widetilde{\Lambda}_l(k)\setminus\mathcal{F}_{l+1}(k)\big)\cup
\mathcal{G}_{l+1}(k).
 \]
Then we have $\widetilde{\Lambda}_{l+1}(k)\setminus\mathcal{G}_{l+1}(k)\subset\widetilde{\Lambda}_l(k)$. Hence, by (c2), $\lambda([\widehat{\tau}])
 <\eta^k$ for every $\widehat{\tau}\in\widetilde{\Lambda}_{l+1}(k)\setminus\mathcal{G}_{l+1}(k)$. Further, By (c2) and the argument in Lemma \ref{lem3}, for $\widehat{\tau}\in \mathcal{G}_{l+1}(k)$ and $\widetilde{\sigma}\in \Gamma_l(k)$ (cf. (\ref{temp1}) and (\ref{tem2})), we have
 \[
 \lambda([\widetilde{\sigma}])<\eta^k\leq\lambda([\widehat{\tau}^-]);\;\lambda([\widehat{\tau}])
 \leq\lambda([\widetilde{\omega}])<\eta^k.
 \]
 This implies that $\widetilde{\sigma}\nprec\widehat{\tau}$. Since $|\widehat{\tau}|>|\widetilde{\sigma}|$, we obtain that $\widetilde{\sigma},\widehat{\tau}$ are incomparable. As in the proof of Lemma \ref{lem6} (b3), by the definition of $\mathcal{G}_{l+1}(k)$, for every pair
 \[
 \widehat{\tau}\in(\widetilde{\Lambda}_{l+1}(k)\cap\Phi_{\xi_{l+1}(k)})\setminus\mathcal{G}_{l+1}(k),\;\;\widehat{\rho}\in\Gamma_l(k)\cup\mathcal{G}_{l+1}(k)
  \]
  $\widehat{\tau},\widehat{\rho}$ are incomparable. Combining the above analysis with (c3), we obtain that the words in  $\widetilde{\Lambda}_{l+1}(k)\cap\big(\bigcup_{h=\xi_1(k)}^{\xi_{l+1}(k)}\Phi_h\big)$ are pairwise incomparable. Thus, (c1)-(c3) hold with $l+1$ in place of $l$.

By induction, we obtain sets $\mathcal{F}_M(k), \widetilde{\mathcal{F}}_M(k), \mathcal{G}_M(k)$ and $\widetilde{\Lambda}_M(k)$ such that (c1)-(c3) are fulfilled for $l=M$. One can see that
\[
\widetilde{\Lambda}_M(k)=\bigg(\widetilde{\Lambda}_1(k)\setminus\bigcup_{l=2}^M\mathcal{F}_l(k)\bigg)\cup\bigg(\bigcup_{l=2}^M\mathcal{G}_l(k)\bigg).
\]

\begin{lemma}\label{antichain}
$\widetilde{\Lambda}_M(k)$ is a finite maximal antichain in $W$.
\end{lemma}
\begin{proof}
By the construction of $\widetilde{\Lambda}_M(k)$, we know that, the words in $\widetilde{\Lambda}_M(k)$ with length not exceeding $\xi_M(k)$, are pairwise incomparable. Now, if $\xi_M(k)<\overline{\xi}(k)$, then for all $\xi_M(k)+1\leq h\leq\overline{\xi}(k)$, we have $\ell(h)=\ell(\xi_M(k))$ and
 \[
\widetilde{\Lambda}_M(k)\cap\bigg(\bigcup_{h=\xi_M(k)+1}^{\overline{\xi}(k)}\Phi_h\bigg)\subset\widetilde{\Lambda}_1(k).
\]
Using the same argument as that in the proof for Claim 1, one can see that the words in $\widetilde{\Lambda}_M(k)$ are pairwise incomparable.

 By the definitions of $\mathcal{F}_h(k)$ and $\mathcal{G}_h(k)$, for every $\widetilde{\sigma}\in \mathcal{F}_h(k)$ and $2\leq h\leq M$, we have (see Lemmas \ref{lem2} and \ref{lem3}),
 \begin{equation}\label{s7}
\sum_{\widetilde{\rho}\in\mathcal{F}(\widehat{\sigma})}\lambda(\widehat{\rho})=\sum_{\widehat{\tau}\in\mathcal{G}(\widetilde{\sigma})} \lambda([\widehat{\tau}]).
 \end{equation}
 It follows that
 \[
 \sum_{\widehat{\tau}\in\widetilde{\Lambda}_M(k)}\lambda([\widehat{\tau}])=\sum_{\widehat{\tau}\in\widetilde{\Lambda}_1(k)}\lambda([\widehat{\tau}])
 =\sum_{\widehat{\tau}\in\widetilde{\Lambda}_1(k)}\mu(F_{\mathcal{L}^{-1}(\widehat{\tau})})=1.
 \]
Suppose that $\widehat{\rho}\notin\bigcup_{\widehat{\tau}\in\widetilde{\Lambda}_M(k)}[\tau]$ for some $\widehat{\rho}\in W$. Then by Remark \ref{rem3}, there exists some $\widehat{\zeta}\in\Phi^*$ with $|\widehat{\zeta}|>\overline{\xi}(k)$ such that $\widehat{\rho}\in[\widehat{\zeta}]$ and $[\widehat{\zeta}]\cap[\widehat{\tau}]=\emptyset$ for all $\widehat{\tau}\in\widetilde{\Lambda}_M(k)$; because, otherwise, we would have $\widehat{\rho}\in[\widehat{\zeta}]\subset\bigcup_{\widehat{\tau}\in\widetilde{\Lambda}_M(k)}[\tau]$. This implies that $\sum_{\widehat{\tau}\in\widetilde{\Lambda}_M(k)}\lambda([\widehat{\tau}])<1$, which is a contradiction. Thus, we obtain
\[
W=\bigcup_{\widehat{\tau}\in\widetilde{\Lambda}_M(k)}[\tau].
\]
Thus, we conclude that $\widetilde{\Lambda}_M(k)$ is a finite maximal antichain.
\end{proof}
 Using the following lemma, we establish an estimate for the difference that is caused by the replacements of $\mathcal{F}(\widetilde{\sigma})$ with $\mathcal{G}(\widetilde{\sigma})$ for $\widetilde{\sigma}\in\mathcal{F}_h(k)$ and $2\leq h\leq M$.
\begin{lemma}\label{lem4}
There exists a constant $C_1$ such that, for every $\widetilde{\sigma}\in\mathcal{F}_h(k)$,
\[
\bigg|\sum_{\widetilde{\omega}\in\mathcal{F}(\widetilde{\sigma})}\lambda([\widetilde{\omega}])\log\lambda([\widetilde{\omega}])
-\sum_{\widehat{\tau}\in\mathcal{G}(\widetilde{\sigma})}\lambda([\widehat{\tau}])\log\lambda([\widehat{\tau}])\bigg|\leq C_1\sum_{\widetilde{\omega}\in\mathcal{F}(\widetilde{\sigma})}\lambda([\widetilde{\omega}]).
\]
\end{lemma}
\begin{proof}
Let $\widetilde{\sigma}\in\mathcal{F}_h(k)$ be as given in (\ref{s4}). We write
\[
\sigma^\sharp:=((i_1,j_1),\ldots,(i_{\ell(\xi_2(k))-1},j_{\ell(\xi_2(k))-1})\times(j_{\ell(\xi_2(k))+1},\ldots,j_{\xi_2(k)-1}).
\]
By the definition of $\mathcal{F}(\widetilde{\sigma})$ and that of the measure $\lambda$, we have
\begin{eqnarray}\label{t1}
&&\sum_{\widetilde{\omega}\in\mathcal{F}(\widetilde{\sigma})}\lambda([\widetilde{\omega}])\log\lambda([\widetilde{\omega}])\nonumber\\
&&=\sum_{i\in G_x(j_{\ell(\xi_2(k))})}(\lambda([\sigma^\sharp])p_{ij_{\ell(\xi_2(k))}}q_{j_{\xi_2(k)}})\log(\lambda([\sigma^\sharp])
p_{ij_{\ell(\xi_2(k))}}q_{j_{\xi_2(k)}})\nonumber\\
&&=q_{j_{\xi_2(k)}}q_{j_{\ell(\xi_2(k))}}(\lambda([\sigma^\sharp])\log(\lambda([\sigma^\sharp])\nonumber\\&&\;\;+
\lambda([\sigma^\sharp])q_{j_{\xi_2(k)}}\sum_{i\in G_x(j_{\ell(\xi_2(k))})}p_{ij_{\ell(\xi_2(k))}}\log p_{ij_{\ell(\xi_2(k))}}\nonumber\\&&\;\;+
\lambda([\sigma^\sharp])q_{j_{\xi_2(k)}}q_{j_{\ell(\xi_2(k))}}\log q_{j_{\xi_2(k)}}.
\end{eqnarray}
In an analogous manner, we have
\begin{eqnarray}\label{t2}
\sum_{\widehat{\tau}\in\mathcal{G}(\widetilde{\sigma})}\lambda([\widehat{\tau}])\log\lambda([\widehat{\tau}])&&=
\frac{}{}q_{j_{\ell(\xi_2(k))}}q_{j_{\xi_2(k)}}(\lambda([\sigma^\sharp])\log(\lambda([\sigma^\sharp])\nonumber\\&&\;\;+
\lambda([\sigma^\sharp])q_{j_{\ell(\xi_2(k))}}\sum_{i\in G_x(j_{\xi_2(k)})}p_{ij_{\xi_2(k)}}\log p_{ij_{\xi_2(k)}}\nonumber\\
&&\;\;+\lambda([\sigma^\sharp])q_{j_{\ell(\xi_2(k))}}q_{j_{\xi_2(k)}}\log q_{j_{\ell(\xi_2(k))}}.
\end{eqnarray}
Further, one can observe that
\begin{eqnarray}\label{t3}
\lambda(\mathcal{F}(\widetilde{\sigma})):=\sum_{\widetilde{\omega}\in\mathcal{F}(\widetilde{\sigma})}\lambda([\widetilde{\omega}])
=q_{j_{\xi_2(k)}}q_{j_{\ell(\xi_2(k))}}\lambda([\sigma^\sharp])\geq\underline{q}^2\lambda([\sigma^\sharp]).
\end{eqnarray}
Hence, $\lambda([\sigma^\sharp])\leq\underline{q}^{-2}\lambda(\mathcal{F}(\widetilde{\sigma}))$. Set
$C_0:=-2\overline{q}^2\log\underline{p}$. By (\ref{t1})-(\ref{t3}), we obtain
\begin{eqnarray*}
\bigg|\sum_{\widetilde{\omega}\in\mathcal{F}(\widetilde{\sigma})}\lambda([\widetilde{\omega}])\log\lambda([\widetilde{\omega}])
-\sum_{\widehat{\tau}\in\mathcal{G}(\widetilde{\sigma})}\lambda([\widehat{\tau}])\log\lambda([\widehat{\tau}])\bigg|\leq \underline{q}^{-2}C_0\lambda(\mathcal{F}(\widetilde{\sigma})).
\end{eqnarray*}
The lemma follows by defining $C_1:=\underline{q}^{-2}C_0$.
\end{proof}
 We are now able to determine the asymptotic order for $(s_k)_{k=1}^\infty$. We have
\begin{lemma}\label{lem5}
Let $s_k$ be as defined in (\ref{sk}). Then we have $|s_k-s_0|\lesssim k^{-1}$.
\end{lemma}
\begin{proof}
By the construction of $\widetilde{\Lambda}_M(k)$, we have
\begin{eqnarray}\label{s6}
\widetilde{\Lambda}_M(k)=\bigg(\widetilde{\Lambda}_k\setminus
\bigcup_{h=2}^M\bigcup_{\widetilde{\sigma}\in\widetilde{\mathcal{F}}_h(k)}\mathcal{F}(\widetilde{\sigma})\bigg)\cup
\bigg(\bigcup_{h=2}^M\bigcup_{\widetilde{\sigma}\in\widetilde{\mathcal{F}}_h(k)}G(\widetilde{\sigma})\bigg).
\end{eqnarray}
By the definition of $G(\widetilde{\sigma})$, we have $|\widehat{\tau}|=|\widetilde{\sigma}|$ for all $\widehat{\tau}\in G(\widetilde{\sigma})$ and $\widetilde{\sigma}\in\mathcal{F}(\widetilde{\sigma})$.
Thus, by(\ref{s7}), for every $2\leq h\leq M$ and $\widetilde{\sigma}\in\mathcal{F}_h(k)$, we have
\begin{eqnarray*}
\sum_{\widetilde{\rho}\in\mathcal{F}_h(k)}\lambda([\widetilde{\rho}])\log m^{-|\widetilde{\rho}|}=\sum_{\widehat{\tau}\in\mathcal{G}_h(k)}\lambda([\widehat{\tau}]\log m^{-|\widehat{\tau}|}.
\end{eqnarray*}
This, along with (\ref{s6}), yields
\begin{eqnarray}\label{s8}
\sum_{\sigma\in\widetilde{\Lambda}_M(k)}\lambda([\widehat{\tau}])\log m^{-|\widehat{\tau}|}=\sum_{\widetilde{\sigma}\in\widetilde{\Lambda}_k}\lambda([\widetilde{\sigma}]\log m^{-|\widetilde{\sigma}|}.
\end{eqnarray}
Also, by (\ref{s6}), we have
\begin{eqnarray*}
\sum_{\widehat{\tau}\in\widetilde{\Lambda}_M(k)}\lambda([\widehat{\tau}])\log \lambda([\widehat{\tau}])&=&\sum_{\widehat{\tau}\in\widetilde{\Lambda}_k\setminus\bigcup_{h=2}^M\mathcal{F}_h(k)}\lambda([\widehat{\tau}])\log \lambda([\widehat{\tau}])\\&&+\sum_{h=2}^{M}\sum_{\widetilde{\sigma}\in\widetilde{\mathcal{F}}_h(k)}\sum_{\widehat{\tau}\in G(\widetilde{\sigma})}\lambda([\widehat{\tau}])\log \lambda([\widehat{\tau}]).
\end{eqnarray*}
Using this and Lemma \ref{lem4}, we deduce
\begin{eqnarray*}
\Delta_k:&=&\bigg|\sum_{\widehat{\tau}\in\widetilde{\Lambda}_M(k)}\lambda([\widehat{\tau}])\log \lambda([\widehat{\tau}])-\sum_{\widetilde{\sigma}\in\widetilde{\Lambda}_k}\lambda([\widetilde{\sigma}])\log \lambda([\widetilde{\sigma}])\bigg|\\
&=&\bigg|\sum_{h=2}^{M}\sum_{\widetilde{\sigma}\in\widetilde{\mathcal{F}}_h(k)}\sum_{\widehat{\tau}\in G(\widetilde{\sigma})}\lambda([\widehat{\tau}])\log \lambda([\widehat{\tau}])-\sum_{h=2}^{M}\sum_{\widetilde{\sigma}\in\widetilde{\mathcal{F}}_h(k)}\sum_{\widetilde{\omega}\in \mathcal{F}(\widetilde{\sigma})}\lambda([\widetilde{\omega}])\log \lambda([\widetilde{\omega}])\bigg|
\\&\leq&\sum_{h=2}^{M}\sum_{\widetilde{\sigma}\in\widetilde{\mathcal{F}}_h(k)}\bigg|\sum_{\widehat{\tau}\in G(\widetilde{\sigma})}\lambda([\widehat{\tau}])\log \lambda([\widehat{\tau}])-\sum_{\widetilde{\omega}\in \mathcal{F}(\widetilde{\sigma})}\lambda([\widetilde{\omega}])\log \lambda([\widetilde{\omega}])\bigg|\\
&\leq&C_1\sum_{h=2}^{M}\sum_{\widetilde{\sigma}\in\widetilde{\mathcal{F}}_h(k)}\lambda(\mathcal{F}(\widetilde{\sigma}))\leq C_1.
\end{eqnarray*}
Note that $l(\widetilde{\Lambda}_M(k))=\underline{\xi}(k)$ (cf. (\ref{temp03})). Thus, by (\ref{equivalent}), (\ref{s8}) and (\ref{xikcomparable}), we obtain
\begin{eqnarray}\label{t4}
\big|s_k-t(\widetilde{\Lambda}_M(k))\big|&=&\bigg|\frac{\sum_{\sigma\in\Lambda_k}\mu_\sigma\log\mu_\sigma}{\sum_{\sigma\in\Lambda_k}\mu_\sigma\log m^{-|\sigma|}}-
\frac{\sum_{\widehat{\tau}\in\widetilde{\Lambda}_M(k)}\lambda([\widehat{\tau}])\log\lambda([\widehat{\tau}])}
{\sum_{\sigma\in\widetilde{\Lambda}_M(k)}\lambda([\widehat{\tau}])\log m^{-|\widehat{\tau}|}}\bigg|\nonumber
\\&=&\frac{\Delta_k}
{\bigg|\sum_{\sigma\in\widetilde{\Lambda}_M(k)}\lambda([\widehat{\tau}])\log m^{-|\widehat{\tau}|}\bigg|}\lesssim\frac{1}{\underline{\xi}(k)}\lesssim\frac{1}{k}.
\end{eqnarray}
On the other hand, by Lemmas \ref{lem01} and \ref{antichain}, we deduce
\begin{equation}\label{t5}
|t(\widetilde{\Lambda}_M(k))-s_0|\lesssim k^{-1}.
\end{equation}
Combining (\ref{t4}) and (\ref{t5}), we conclude that
\[
|s_k-s_0|\leq|s_k-t(\widetilde{\Lambda}_M(k))|+|t(\widetilde{\Lambda}_M(k))-s_0|\lesssim k^{-1}.
\]
This completes the proof of the lemma.
\end{proof}

\emph{Proof  of Theorem \ref{mthm}}
Let $R^{s_0}_k(\mu)$ be as defined in Lemma \ref{sz1}. As a consequence of Proposition \ref{pre01}, we have
\[
R^{s_0}_k(\mu)\approx s_0^{-1}\log\phi_k+s_k^{-1}\sum_{\sigma\in\Lambda_k}\mu_\sigma\log\mu_\sigma.
\]
By (\ref{lambdak}) and (\ref{temp01}), we know that for every $\sigma\in\Lambda_k$, we have $\log\mu_\sigma^{-1}\approx\log\phi_k$. Note that the sequence $(s_k)_{k=1}^\infty$ is bounded away from zero. Hence, we obtain
\[
R^{s_0}_k(\mu)\approx \big(s_0^{-1}-s_k^{-1}\big)\log\phi_k.
\]
From this, (\ref{xikcomparable}) and Lemma \ref{lem5}, we deduce
\begin{eqnarray*}
\big|R_k^{s_0}(\mu)\big|\approx\big|\big(s_0^{-1}-s_k^{-1}\big)\log\phi_k\big|\asymp 1.
\end{eqnarray*}
Thus, by Lemma \ref{sz1}, we conclude that $0<\underline{Q}_0^{s_0}(\mu)\leq\overline{Q}_0^{s_0}(\mu)<\infty$.


\end{document}